\documentclass[11pt]{amsart} 

\usepackage{stmaryrd}
\usepackage{epigraph}
\usepackage[margin=1.5in]{geometry}
\usepackage{amsmath}
\usepackage{amsfonts}
\usepackage{amssymb}
\usepackage{xcolor}
\usepackage{graphicx} 
\usepackage[T1]{fontenc}

\makeatletter
\renewcommand{\@seccntformat}[1]{ {\csname
the#1\endcsname}.\hspace{0.2em}}
\makeatother

\def \P {{\mathbb P}}

\def\E{\mathbb E}
\def \R {{\mathbb R}}
\def \N {{\mathbb N}}
\def \Z {{\mathbb Z}}

\begin{document}
\theoremstyle{plain}
\newtheorem{axiom}{Axiom}
\newtheorem{claim}[axiom]{Claim}
\newtheorem{theorem}{Theorem}[section]
\newtheorem{lemma}[theorem]{Lemma}
\newtheorem{prop}[theorem]{Proposition}
\newtheorem{coro}[theorem]{Corollary}
\theoremstyle{remark}
\newtheorem{definition}[theorem]{Definition}
\newtheorem*{example}{Example}
\newtheorem*{fact}{Fact}

\definecolor{gris25}{gray}{0.75}
 
\newcounter{encart}
\renewcommand{\theencart}{\thechapter.\arabic{encart}}
\newenvironment{encart}[1]
{
\refstepcounter{encart}
\small \sf \medskip
\noindent\colorbox{gris25}{
\makebox[\textwidth][c]{{\sf E{\scriptsize{NCART}}  
\theencart} -- \bfseries #1}}\medskip 
\addcontentsline{toc}{subsection}{\sf \small \hspace{0.85 cm} Encart \theencart. #1}
}
{\smallskip
\noindent\colorbox{gris25}{\makebox[\textwidth][c]{\hspace{1cm}}}
\rm \normalsize \medskip}

\newenvironment{theobis}[1]
  {\renewcommand{\thetheo}{\ref{#1}$\ bis $}%
   \addtocounter{theo}{-1}%
   \begin{theo}}
  {\end{theo}}
  \title[A converse to Pitman's theorem  for   $A_1^1$]
 {A converse to Pitman's theorem  for a space-time Brownian motion in  a type $A_1^1$ Weyl chamber}
 \author{Manon Defosseux and Charlie Herent } 
\address{M. D. : Universit\'e Paris Cit\'e, CNRS, MAP5, F-75006 Paris, France}\email{manon.defosseux@parisdescartes.fr}
\address{C. H. : Universit\'e Paris Cit\'e, CNRS, MAP5, F-75006 Paris, France and Universit\'e Paris-Est, CNRS, Institut Gaspard Monge, Champs-Sur-Marne, France}\email{charlie.herent@ens-rennes.fr} 
  \maketitle  
     \begin{abstract} We prove an inverse Pitman's theorem for a space-time  Brownian motion  conditioned in Doob's sense to remain in an affine Weyl chamber. Our theorem provides  a way to recover an unconditioned space-time Brownian motion from a conditioned  one by applying a sequence of path transformations. 
  \end{abstract}
  \setcounter{tocdepth}{1}
  \tableofcontents
  \section{Introduction}
 
 Let $\{b(t),\,t\ge 0\}$ be  a real brownian motion then Pitman's theorem \cite{pitman} asserts that 
$$\mathcal Pb(t)=b_t-2\inf_{0\leq s \leq t}b_s, \, t\ge 0,$$ is a Bessel process of dimension 3, which has the same distribution as a brownian motion conditioned, in Doob's sense, to remain in the positive half-line.
 This seminal result  has given rise to many generalizations or variations, see for instance \cite{bertoin,biane1,BJ,RC,MY,OY,RVY}.
Let us briefly describe one of the most accomplished one, due to Ph. Biane, Ph. Bougerol and N. O'Connell \cite{bbo,bbobis}. In Pitman's theorem, the unconditioned Brownian motion lives on $\R$ and the conditioned one  lives on $\R_+$. Actually $\R_+$ can be seen as the fundamental chamber of  the   group generated by the reflection through $0$ acting on $\R$. This group is the simplest one among the class of   Coxeter groups. In \cite{bbo} the authors have shown how to obtain a brownian motion conditioned to remain in the fundamental chamber of a finite Coxeter group by applying a sequence of Pitman  type transformations associated to a set of generators of the Coxeter group, according to the order of appearance of the generators in a reduced  decomposition of the longest element in the group. This paper has 
brought to light    deep connections between the Pitman transform and the Littelmann path model \cite{littel2} which is a combinatorial model  that describes the representations   of  a Kac--Moody Lie algebra. 

The  affine Coxeter group of type $A_1^1$ is the Weyl group of a rank one affine Kac--Moody algebra.  In   \cite{boubou-defo}    another Pitman type theorem has been established for a conditioned random process living in the fundamental chamber of the latter group, whose interior is the subset $C_{\mbox{aff}}$ defined below.    Pitman's theorem in that case involves two Pitman  type transformations corresponding to the generators of the   group and  is only asymptotic. Since there is no longest element in that case  one has to apply an infinite number of transformations. Moreover,  quite surprisingly,  the conditioned process is not obtained by applying successively and infinitely the two Pitman transforms to an unconditioned process  : a   correction has to be applied, which involves two L\'evy type transformations.   
 
 One can formulate a converse to Pitman's theorem, indeed given for $T\ge 0$ a nonnegative  continuous  real trajectory  $\{\pi(t), t\in[0,T]\}$ starting at $0$, and a real number $x\in [0,\pi(T)]$, there is a unique     real trajectory $\eta$ starting at $0$  such  that
$$\mathcal P\eta=\pi\, \textrm{ and } \,x=-\inf_{0\le s\le T}\eta(s).$$
It satisfies $\eta(t)=\pi(t)-2\min(x,\inf_{t\le s\le T}\pi(s))$, $t\in [0,T]$. In other words, a path defined on $[0,T]$ is entirely determined by  its image by the Pitman transform and a real number that  we will call  a string coordinate, according to the terminology of   Littelmann.   It follows that one can construct a standard real Brownian motion starting from a Bessel 3 process and a suitable real random variable.  This construction generalizes to the case of finite Coxeter groups \cite{bbo}.

We propose  to give an analog of this recontruction   for the case of the conditioned Brownian motion of \cite{boubou-defo}. Let us nevertheless notice that  our reconstruction is of a very different nature from the one previously described in the context of a finite Coxeter group. In the latter case  actually the   reconstruction is a direct consequence of a deterministic result, whereas our result is a purely probabilistic one.  This is a reconstruction in law.

 We use results obtained in   \cite{boubou-defo} but  our approach is quite different  from the one adopted in this last paper.  Indeed the proof of  Pitman's theorem in \cite{boubou-defo} relies on some approximations of Brownian motions in the fundamental Weyl chamber of the affine Coxeter group in type $A_1^1$ by Brownian motions in fundamental chambers of dihedral groups and the version of Pitman's theorem for these groups established in \cite{bbobis}. Instead we use approximations by  random walks defined using the Littelmann path model for the affine Kac-Moody algebra $A_1^1$. Such random walks have been originally introduced by C. Lecouvey, E. Lesigne and M. Peign\'e  in  \cite{LLP}.

  It  has been proved in \cite{defo2} (see also \cite{defo4}) that  these last processes  can also be approximated by   random walks defined using the Littelmann path model for the affine Kac-Moody algebra $A_1^1$. Such random walks have been originally introduced by C. Lecouvey, E. Lesigne and M. Peign\'e  in  \cite{LLP}.  These are the approximations   we use here.  Their laws offer the advantage of being given by explicit formulas coming from representation theory, which allows to make computations. This is a huge advantage and makes our paper fall in the large category of the so-called integrable probability.

  Demazure crystals play a crucial   role in our paper.   These crystals have beautiful combinatorial properties. Nevertheless, as far as we know, they haven't been used before in the framework of integrable probability, which maybe can be explained  by the fact that they do not form a tensor category, so that they do not define an hypergroup structure which could naturally relate them    to a   Markov process  in a usual way (see for instance  \cite{Wildberger} and references therein).  Since  the Littelmann model and Demazure character formulas  that   we use      are available for any affine Kac-Moody algebra,   they might be useful for obtaining an inverse  Pitman's theorem in a more general context.

Let us make a last remark about our result. Actually, in the context of a finite Coxeter group, one can state another reconstuction theorem. In the simplest case, it states that if $\{r_t, t\ge 0\}$ is a Doob-conditioned positive standard Brownian motion  then 
$$\{r_t-2\inf_{s\ge t}r_s, t\ge 0\}$$
is a real standard Brownian motion (see \cite{revuzyor}, chapter VI, corollary 3.7). More generally, for any finite Coxeter group,  there exists such a functional transformation, which sends a conditioned Brownian motion to an unconditioned one. Such a result seems to be unattainable   for $A_1^1$. Actually, in the finite case,  the string coordinates of a Brownian motion are infinite and a Brownian motion  stands morally for the lowest weight path in the Littelmann module of  a Verma module. There is no such a lowest weight path in the case of $A_1^1$.
   
   The paper is organized as follows. In section \ref{statement} we give a statement of an inverse Pitman's theorem for $A_1^1$.   In section \ref{The affine Lie algebra} we briefly recall the necessary background on representation theory of the affine Lie algebra $A_1^1$. The Littelmann path model for a Kac-Moody algebra $A_1^1$ and its connection with   Pitman transforms is explained in  section \ref{Pitman transforms and Littelmann modules}. We define  in section \ref{RWLP} random walks with increments in a Littelmann   module   and the associated random processes in the affine Weyl chamber. These processes can be seen as  approximations of the unconditioned and conditioned Brownian  motions   introduced in section \ref{The continuous counterpart}. Finally we prove an inverse Pitman's theorem for $A_1^1$ in section \ref{An inverse Pitman's theorem}. 
   \bigskip

    {\it Acknowledgments:}  This project is supported  by the Agence Nationale de la Recherche funding CORTIPOM  ANR-21-CE40-0019.

   \section{Statement of the   theorem}\label{statement} 
   For a real $x\ge 0$, we define two  functional transformations $I^x_0$  and  $I_1^x$ acting on  the set of continuous maps $\eta:\R_+\to \R^2$ such that $\eta(t)=(t,f(t))$, where $f(t)\in \R$, for $t\ge 0$, and 
   $\lim_{t\to\infty}f(t)/t\in (0,1)$ as 
   \begin{align*}
   I^x_0\eta(t)&=\big(t,f(t)+2\min(x,\inf_{s\ge t}(s-f(s)))\big), \\
   I^x_1\eta(t)&=\big(t,f(t)-2\min(x,\inf_{s\ge t}(f(s)))\big), \quad t\ge 0.
   \end{align*}
  Let $\{B(t)=(t,b_t+t/2), t\ge 0\}$ be  a space-time Brownian motion,  where $b$ is a standard Brownian motion, and a   space-time Brownian motion   $\{A(t)=(t,a_t), t\ge 0\}$  with  a drift $1/2$, conditioned  to remain in the domain $C_{\mbox{aff}}$  defined by
   $$C_{\mbox{aff}}=\{(t,x)\in \R_+\times \R:  0<x<t\}.$$
 See section \ref{The continuous counterpart} for the definition of this process. Let $\varepsilon_n, n \geq0$ be a sequence  of independent exponential random variables with parameter $1$. Let $p \in \mathbb{N}$, define  $\xi_{0,p}(\infty)=\varepsilon_0 ,$ and, for all $k \in\{1,\dots,p\}$,
$$\frac{\xi_{k,p}(\infty)}{k}=\sum_{n=k}^{p} \frac{ 2\varepsilon_n}{n(n+1)}.$$ 
 The notational choices will be hopefully clearer  later. Then one has the following reconstruction theorem.

   \begin{theorem}\label{reconstruction}
The sequence of processes $$\{I_{0}^{\xi_{0,p}(\infty)}\dots I_{p}^{\xi_{p,p}(\infty)}A(t), t\ge 0\}, \, p\ge 0,$$
converges,  in the sense of finite dimensional distributions,  towards the space-time Brownian motion $\{B(t),t\ge 0\}$.
\end{theorem}
 This theorem is a converse to Theorem 7.1 in \cite{boubou-defo}.  Let us notice that there is no correction term here. Actually the correction term in the   Pitman's Theorem proved in \cite{boubou-defo} comes from the fact that the sequence of  string coordinates associated to a Brownian motion is a convergent sequence with limit $2$.  The   law of the random sequence  in  the previous theorem  is the law of   the  string coordinates conditioned to be ultimately equal to $0$.   So this is not a surprise that no correction term is needed for this reconstruction theorem.
         \section{The affine Lie algebra $A_1^{1}$ and its representations}\label{The affine Lie algebra}
    We recall some standard facts about  the affine Lie algebra  of type $A_1^{1}$. See   \cite{Kac} for a presentation of affine Lie   algebras and their representations.  For our purpose, we only need to define and consider  a realization of a real Cartan subalgebra. Let $\mathfrak h_\R$ and are $\mathfrak h_\R^*$ two copies of $\R^3$ in standard duality. One has
$$\mathfrak h_\R=\mbox{Span}_{\mathbb R}\{c, {\alpha}^\vee_1,d\},\, \, \mathfrak h_\R^*=\mbox{Span}_{\mathbb R}\{\Lambda_0,\alpha_1,\delta\},$$  
 where $c=(1,0,0), \alpha^\vee_1=(0,1,0),  d=(0,0,1),$ and  $\Lambda_0=(1,0,0)$, $\alpha_1=(0,2,0)$, $\delta=(0,0,1)$ in $\R^3$.

 Let $\alpha^\vee_0=(1,-1,0)$ and $\alpha_0=(0,-2,1)$, so that $c=\alpha^\vee_0+\alpha^\vee_1$ and  $\delta=\alpha_0+\alpha_1$.  The vectors $\alpha_0$ and $\alpha_1$ are  the two positive simple roots of $A_1^1$ and $\alpha^\vee_0$ and $ \alpha^\vee_1$ their coroots. We denote by $\langle \cdot,\cdot\rangle$  the natural pairing. 
 The set of integral weights is
$$P=\{\lambda\in \mathfrak h_\R^*: \langle\lambda,\alpha^\vee_i\rangle \in \mathbb Z, i=0,1\},$$ 

and the set of dominant integral weights 
$$P_+=\{\lambda\in \mathfrak h_\R^*: \langle \lambda,\alpha^\vee_i\rangle\in \mathbb N, i=0,1\}.$$ 
\paragraph{\bf Highest weight representations.} For  a dominant integral weight $\lambda$, the character of the irreducible  representation $V(\lambda)$ of  $A_1^1$  with highest weight $\lambda$ is defined as the formal series    
\begin{align}\label{weightdecomp} 
\mbox{ch}_\lambda=\sum_{\beta\in P}\mbox{dim} (V(\lambda)_\beta)e^{\beta}, 
\end{align}
where $V(\lambda)_\beta$ is the weight space  corresponding to the weight $\beta$ in $V(\lambda)$. Let $e^\beta(h)=e^{ \langle \beta,h\rangle  }$ for $h\in \mathfrak{h}_\R$. The series converges absolutely if $ \langle \delta,h\rangle>0$ otherwise it diverges.
  The character can be extended to the set of   $h\in \mathfrak h_\R\oplus i \mathfrak h_\R$ such that $\Re\langle \delta,h\rangle>0$.
  The   Weyl group $W$ is    the group  generated by the reflections  $s_{\alpha_{i}}$, for $i\in \{0,1\}$,   defined on $\mathfrak{h}^*_\R$ by $$s_{\alpha_{i}}(\beta)=\beta-\langle \beta,\alpha^\vee_i\rangle\alpha_i, \, \beta\in \mathfrak h_\R^*.$$ Weyl's character formula (chapter 10 of \cite{Kac}) states  that 
\begin{align}\label{Weyl}
\mbox{ch}_\lambda=\frac{\sum_{w\in W}\det(w)e^{w(\lambda+\rho)-\rho}}{\prod_{\alpha\in R_+}(1-e^{-\alpha}) },
\end{align}
where $\det(w)$ is the determinant of the linear map $w$, $\rho=2\Lambda_0+\frac{\alpha_1}{2}$ and  $R_+$ is the set of positive roots defined by $$\label{rac_pos}R_+=\{\alpha_0+n\delta, \alpha_1+n\delta, (n+1)\delta, n\in \N\}.$$  In particular 
\begin{align}\label{ParticularWeyl}
\prod_{\alpha\in R_+}(1-e^{-\alpha})=\sum_{w\in W}\det(w)e^{w(\rho)-\rho}.
\end{align}
The affine Weyl group  W is the semi-direct product $T\ltimes W_0$ where $W_0$ is the subgroup generated by $s_{\alpha_1}$ and $T$ is the subgroup of transformations $t_{k}$, $k\in \Z$, defined by 
$$t_k(\lambda)=\lambda+k(\lambda,\delta)\alpha_1-(k(\lambda,\alpha_1)+k^2(\lambda,\delta))\delta, \,\, \lambda\in \mathfrak{h}^*.$$
Thus   for $\lambda\in P_+$, one has
$$\sum_{w\in W}\det(w)e^{w(\lambda+\rho)}=\sum_{k\in \Z}e^{t_k(\lambda+\rho)}-e^{t_ks_{\alpha_1}(\lambda+\rho)},$$
and for $\lambda=n\Lambda_0+m\frac{\alpha_1}{2}$, with $(m,n)\in \N^2$ such that $0\le m\le n$, $a\in \mathbb R$, and $b>0$, the Weyl--Kac character formula becomes  here 
\begin{align}\label{charA11}
\mbox{ch}_{\lambda}(a\alpha^\vee_1+bd) =\frac{\sum_{k\in \mathbb Z}\sinh(a(m+1)+2ak(n+2))e^{-b(k(m+1)+k^2(n+2))}}{\sum_{k\in \mathbb Z}\sinh(a+4ak)e^{-b(k+2k^2)}}.
\end{align}  

 \paragraph{\bf Verma modules} The  character of a Verma module with highest weight $0$ is denoted by $\mbox{ch}_{M(0)}$. Let us recall some various known expressions of this  character. First of all, one has
 \begin{align}\label{VermaProd}
\mbox{ch}_{M(0)}=\prod_{\alpha\in R_+}(1-e^{-\alpha})^{-1},
\end{align}   One has also
 \begin{align}\label{VermaLim} \mbox{ch}_{M(0)}=\lim_{\langle\lambda,\alpha^\vee_i\rangle\to \infty, i=0;1}e^{-\lambda}\mbox{ch}_{\lambda}  \end{align}
and 
 \begin{align}\label{VermaSum} \mbox{ch}_{M(0)}= (\sum_{w\in W}\det(w)e^{  w(\rho)-\rho})^{-1}, \end{align}
 the last identity being derived from the Weyl character formula.  Note that, for $\lambda\in P+$ and $h\in \mathfrak{h}_\R$ such that $\langle \delta,h\rangle>0$, one has the inequality
 \begin{align}\label{VermaMaj} \mbox{ch}_\lambda(h)e^{-\langle \lambda,h\rangle}\le \mbox{ch}_{M(0)}(h). \end{align}
\section{Pitman transforms and Littelmann modules}\label{Pitman transforms and Littelmann modules}
In this section we explain connections between the Littelmann path model and Pitman transform in the context of the affine Lie algebra $A_1^1$. For more details about the Littelmann path model   see  Peter Littelmann's papers \cite{littel,littel2}. 
  Let $\mathcal C$ be the cone generated by $  P_+$, i.e. $$\mathcal C=\{\lambda\in  {\mathfrak h}^*_\R: \langle\lambda,\alpha_i^\vee\rangle\ge 0, i\in \{0,1\}\}.$$

We fix $T>0$. A path $\pi$ defined on $[0,T]$ is a continuous piecewise linear function $\pi:[0,T]\to  {\mathfrak h}_\R^*$ such that $\pi(0)=0$.  It is called dominant if $\pi(t)\in \mathcal C$ for all $t\in [0,T]$. It is called integral if $\pi(T)\in   P$ and $$\min_{t\in [0,T]}\langle\pi(t),\alpha_i^\vee\rangle\in \Z,\,\textrm{  for   } i\in \{0,1\}.$$    
The Pitman tranforms $\mathcal P_{\alpha_i}$, $i\in \{0,1\}$, are defined on the set on continuous functions $\eta:[0,T]\to \mathfrak h_\R^*$, such that $\eta(0)=0$, by the formula 
  $$\mathcal P_{\alpha_i}\eta(t)=\eta(t)-\inf_{0\le s\le t}\langle \eta(s),\alpha_i^\vee\rangle\alpha_i,\quad t\in [0,T].$$
  Let us notice that  the fact that $\langle\alpha_i,\alpha^\vee_i\rangle=2$  implies that the definition above coincides with  the one of the original Pitman transform. 
   For   a dominant path   $\pi$ defined on $[0,T]$, such that $\pi(T)\in   P_+$, the Littelmann  module $B\pi$ generated by $\pi$ is the set of integral  paths   $\eta$ defined on  $[0,T]$ such that there exists  $k\in \N$ such that $$\mathcal P_{\alpha_{k}}\dots\mathcal P_{\alpha_{0}}\eta=\pi,$$
where $\alpha_{2k}=\alpha_0$ and $\alpha_{2k+1}=\alpha_1$. If $\pi$ is a dominant integral path defined on $[0,T]$ such that $\pi(T)=\lambda\in P_+$, then  the Littelmann path theory ensures that 
\begin{align}\label{char-litt}  \mbox{ch}_{\lambda} =\sum_{\eta\in  B{\pi}}e^{\eta(T)}.
\end{align}
 Moreover for   an integral   path $\eta$ defined on  $[0,T]$ there exists $k_0$ such that for all $k\ge k_0$, one has
$$ \mathcal P_{\alpha_{k}}\dots\mathcal P_{\alpha_{0}}\eta(t)=\mathcal P_{\alpha_{k_0}}\dots\mathcal P_{\alpha_{0}}\eta(t), \quad t\in[0,T]\footnote{It has been proved in \cite{boubou-defo} that this fact remains true if $\eta$  is a continuous, piecewise $C^1$ trajectory in $\mathfrak h_\R^*$.}.$$  
Thus for an integral   path $\eta$ defined on  $[0,T]$, one defines a dominant path $\mathcal{P}\eta$ on $[0,T]$, by   $$\mathcal{P}\eta(t)=\lim_{k\to \infty}\mathcal P_{\alpha_{k}}\dots\mathcal P_{\alpha_{0}}\eta(t), \quad t\in[ 0,T].$$

 \paragraph{\bf String coordinates}    Let $\ell^{(\infty)}(\N)$  be the set of sequences  of nonnegative integers, almost all zero.     Let $\pi$ be a dominant path defined on $[0,T]$ and $\eta\in B\pi$. There exists a unique sequence of nonnegative  integers,    $\mathfrak{a}(\eta):=(a_k)_{k\ge 0}$ almost all zero,  such that
\begin{align}\label{def-string} 
\mathcal P_{\alpha_{m}}\dots\mathcal P_{\alpha_{0}}\eta(T)  =\eta(T)+\sum_{k=0}^ma_k\alpha_{k},\quad m\ge 0.
\end{align} 
Peter Littelmann proved in  \cite{littel2}  that the  map $$\mathfrak{a} : \eta\in B\pi \to \mathfrak{a}(\eta)\in \ell^{(\infty)}(\N)$$  is   injective. The  image of this map, which depends on $\pi$ only through $\pi(T)$, is the set $B(\pi(T))$ defined below. It is the set of vertices of a Kashiwara  crystal  \cite{kash93}. The sets $B(\infty)$ and $B(\lambda)$ defined below  are for instance 
respectively described in \cite{nakash} and \cite{littel2}.   
\begin{definition} \label{crystals} The subset  $B(\infty)$ of $\ell^{(\infty)}(\N)$ is defined as 
\begin{align*} 
B(\infty)&=\{a=(a_k)_{k\ge 0}\in \ell^{(\infty)}(\N): \frac{a_k}{k}\ge \frac{a_{k+1}}{k+1}, k\ge 1\}.
\end{align*}
For  $\lambda\in P_+$, the subset $B(\lambda)$ of  $B(\infty)$   is defined as
\begin{align*}
B(\lambda)&=\{a=(a_k)_{k\ge 0}\in B(\infty) : a_p\le \langle \lambda -\sum_{k=p+1}^\infty a_k\alpha_{k},\alpha_{p}^\vee\rangle, \forall p \ge 0\}\\
&=\{a=(a_k)_{k\ge 0}\in B(\infty) : a_p\le \langle \lambda -\omega( {a})+\sum_{k=0}^p a_k\alpha_{k},\alpha_{p}^\vee\rangle, \forall p \ge 0\},
\end{align*}
where $\omega(  a)=\sum_{k=0}^\infty a_k\alpha_{k}$, which is the opposite of the weight of $  a$ as an element of  the crystal $B(\infty)$  of the Verma module of highest weight $0$.
\end{definition}
Thus identity (\ref{char-litt}) becomes 
\begin{align}\label{char-string}
\mbox{ch}_{\lambda} =\sum_{a\in  B(\lambda)}e^{\lambda-\omega(a)},
\end{align}
and the character of a Verma module is written  with the string coordinates,
 \begin{align}\label{VermaString} \mbox{ch}_{M(0)}=\sum_{a\in B(\infty)}e^{-\omega(a)}. \end{align}
The inverse function of $\mathfrak a$ can be written using the functionals $I_{\alpha_i}^{x,T}$, $i\in\{0,1\}$,  $x\ge 0$,  introduced in \cite{bbo} and defined      by   
$$I_{\alpha_i}^{x,T}f(t)=f(t)-\min(x,\inf_{T\ge s\ge t}\langle f(s), \alpha_i^\vee\rangle) \alpha_i, \quad t\in[0,T],$$
for $f:[0,T]\to \mathfrak{h}^*_\R$. It may be noted that the definition coincides with that given at the beginning of part 2. \\
For $a\in B(\lambda)$ and $\pi$ an integral dominant path on $[0,T]$ such that $\pi(T)=\lambda$,  the only  path $\eta \in B\pi$ such that $\mathfrak a(\eta)=a$ is given by
$$\eta(t)=I_{\alpha_0}^{a_0,T}\dots I_{\alpha_p}^{a_p,T}\pi(t), \quad t\in [0,T],$$
where $p$ is chosen such that $a_k=0$, for all $k\ge p+1$. Notice that if $f$ is a function defined on $\R_+$ with values in $\mathfrak h_\R^*$ such that 
$$\lim_{t\to \infty}\langle f(t),\alpha_i^\vee\rangle=+\infty, \quad i\in\{0,1\},$$
the   definition of $I_{\alpha_i}^{x,T}$, $i\in\{0,1\}$, makes sense for $T=+\infty$. In the following,   we write $I_{\alpha_i}^x$ instead of $I_{\alpha_i}^{x,+\infty}$. We notice that if $f$ is a map with values in $\R\Lambda_0\oplus \R\alpha_1$ then for $t\ge 0$, $i\in \{0,1\}$,
 $$I_{\alpha_i}^xf(t)= I^x_if(t)\mod \delta.$$

\paragraph{ \bf Demazure character}

  One can find   for instance in \cite{kash93}   an introduction to Demazure characters in the context of crystals.  
 For an integer $p\ge 0$ and for $\lambda\in P_+$,  let $w_p=s_{\alpha_p}\dots s_{\alpha_0}$, let  $\pi$ be an integral dominant path defined on  $[0,T]$ such that $\pi(T)=\lambda$,
 and  
 $B^{w_p}\pi=\{\eta\in B\pi: \mathcal P_{\alpha_p}\dots\mathcal P_{\alpha_0}\eta=\pi\}$. One defines $ \mbox{ch}^{w_p}_{\lambda}$ by the formula
    \begin{align}\label{CDLitt}
 \mbox{ch}^{w_p}_{\lambda}=\sum_{\eta\in B^{w_p}\pi} e^{\eta(T)},
 \end{align}
The function  $\mbox{ch}^{w_p}_{\lambda}$ is a Demazure character, i.e. the character of a $U(\frak{n}^+)$-module.
 Written with  the string coordinates, definition (\ref{CDLitt}) becomes
 \begin{align}\label{CDString}\mbox{ch}^{w_p}_{\lambda}=\sum_{a\in B(\lambda),\, a_{p+1}=0} e^{\lambda-\omega(a)}.
 \end{align} 
We define a  Verma--Demazure  character $\mbox{ch}^{w_p}_{M(0)}$ by
\begin{align}\label{VDString}
\mbox{ch}^{w_p}_{M(0)}=\sum_{a\in B(\infty),\, a_{p+1}=0}e^{-\omega(a)}.
\end{align}
 
\section{Random walks and Littelmann paths}  \label{RWLP}
In  this section $m$ is a fixed  positive integer. Let $\pi_{0}$  be the path defined on $[0,1]$ by 
$$\pi_{0}(t)=t\Lambda_0, \quad t\in [0,1],$$
and the Littelmann module $B{\pi_0}$ generated by $\pi_0$.  Let
  $\rho^\vee=2d+\alpha^\vee_1/2$. 
 We fix an integer $m\ge 1$. The formula
\begin{align}\label{loi-discrete}\mu^{m}(\eta)=\frac{e^{\frac{1}{m}\langle\eta(1),\rho^\vee\rangle  }}{\mbox{ch}_{\Lambda_0}(\rho^\vee/m)}, \quad \eta\in B{\pi_0},
\end{align}
defines a probability measure $\mu^{m}$ on $B{\pi_0}$.
Let $(\eta_i^{m })_{i\ge 0}$ be a sequence   of i.i.d random variables with law $\mu^{m}$ and   let $\{\Pi^{m}(t),t\ge 0\}$ be defined by
$$\Pi^{m}(t)=\eta_1^{m}(1)+\dots+\eta_{k-1}^{m}(1)+ \eta_k^{m}(t-k+1),$$ when $t\in [k-1,k[, $ for $k\in \Z_+$. We write $*$ for the usual concatenation of paths, so that for an integer $t$, the restriction of  $\Pi^{m}$ to $[0,t]$ is in $B\pi_0^{*t}$.
For $t\in \N$, let $(\xi_k^m(t))_{k\ge 0}$ be   string coordinates of $\Pi^m\vert_{[0,t]}$. Notice that  the   definition makes sense for $t=\infty$, since each string coordinate is an increasing function of $t$.

 We define a random process $\{ \Pi^m_+(t), t\ge 0\}$ with values in $\mathcal C$  by
$$\Pi^m_+(t)=\mathcal P\Pi^m(t), \quad t\ge 0.$$
 The next proposition follows from the  properties of the  Littelmann path model.   It implies in particular that \\ $\{\Pi_+^m(k), k\ge 
0\}$ is Markovian with transition probabilities given in  Theorem 4.7 of  \cite{LLP}.    It will be very useful in the whole paper as it allows to show that the Markov process $\{   \Pi^m_+(k), k\ge 0\}$ inherits many properties from the random walk $\{\Pi^m(k), k\ge 
0\}$.
\begin{prop}\label{fromMtoC} For any integers $k$ and $n$, and any   fonction $f$ defined on the set of continuous functions $\mathcal C([n,n+k],\R)$,  one has
\begin{align*}&\E\left(f(\Pi_+^m(t):  n\le t\le n+k)\vert \Pi_+^m(s), s\le n\right)\\
&\quad =\E\left(f(\Pi^m(t)+\lambda, 0\le t\le k)\frac{\mbox{ch}_{\Pi^m(k)+\lambda}(\rho^\vee/m)}{\mbox{ch}_\lambda(\rho^\vee/m)}e^{-\langle\Pi^m(k),\rho^\vee/m\rangle}1_{\lambda+\Pi^m\vert_{[0,k]} \in \mathcal C}\right),
\end{align*}
where $\lambda=\Pi_+^m(n)$.
\end{prop}

The next proposition follows from   the fact that the image of a Littelmann module $B\pi$ under $\mathfrak a$  depends on $\pi$ only through the final value of $\pi$.  
\begin{prop} \label{preprop-string} For $u\in \N$ and $f$ a real function defined on $B(\infty)$ one has 
$$\E\left(f(\xi^m(u))\vert  \Pi_+^m(t), t\le u\right)=\frac{\sum_{a\in B(  \Pi_+^m(u))}f( a)e^{\langle \Pi_+^m(u)-\omega(a),\rho^\vee/m\rangle }}{\sum_{a\in B(  \Pi_+^m(u))} e^{\langle \Pi_+^m(u)-\omega(a),\rho^\vee/m\rangle} }$$
where $\xi^m(u)=(\xi_k^m(u))_{k\ge 0}$.
\end{prop}

 \begin{lemma}\label{LGNM} For $i\in \{0,1\}$,  
 $  \langle\Pi^m(k),\alpha^\vee_i\rangle/k $ almost surely converges as $k$  goes to infinity towards  a positive real number.
 
\end{lemma}

\begin{proof}  In a more general context, it has been proved in \cite{LLP}, Proposition 5.4,     that $\E(\eta(1)) $ is the interior of $\mathcal C$. In our particular case,  it is easily   proved using the explicit description of the weights of $V(\Lambda_0)$ given  for instance
in chapter 9 of \cite{HK}. The   convergence follows from a law of large numbers. \end{proof}
The   following lemma is a first useful application of Proposition  \ref{fromMtoC}.
\begin{lemma}\label{LGNC}  For $i\in \{0,1\}$,   in probability,
 $\lim_{k\to\infty}  \langle\Pi_+^m(k),\alpha^\vee_i\rangle =+\infty.$
\end{lemma} 
\begin{proof}
Lemma \ref{LGNM} implies  that  almost surely 
 $\lim_{k\to\infty}  \langle\Pi^m(k),\alpha^\vee_i\rangle =+\infty.$ 
 For $M>0$,  $i\in \{0,1\}$ and $k\ge 1$, Proposition \ref{fromMtoC} gives 
\begin{align*}
\P&\left(\langle\Pi_+^m(k),\alpha_i^\vee\rangle <M\right)=\\
&\quad\quad\E\left(1_{ \{ \langle\Pi^m(k),\alpha_i^\vee\rangle<M\}} \mbox{ch}_{\Pi^m(k)}(\rho^\vee/m)e^{-\langle\Pi^m(k),\rho^\vee/m\rangle}1_{\Pi^m\vert_{[0,k]} \in \mathcal C}\right).
\end{align*}
Upper bound  (\ref{VermaMaj}) and Lemma \ref{LGNM} end  the proof.
\end{proof}

\begin{prop} \label{prop-string} The sequence of string coordinates $\xi^m(\infty)$ is independent of $\{  \Pi_+^m(t),t\ge 0\}$ and 
 $$\P\left(\xi^m(\infty)=a\right)=\frac{e^{-\langle\omega(a),\rho^\vee/m\rangle}}{\mbox{ch}_{M(0)}(\rho^\vee/m)}, \, \, a\in B(\infty).$$
 \end{prop}

\begin{proof} Let $T\ge 0$, $a\in B(\infty)$ and  $f$ be a  real valued    function defined on $B\pi_{0}^{*T}$ that we suppose bounded by $1$. One has
$$\E\left( f(  {\Pi_+^m}_{\vert_{[0,T]}})1_{\{\xi^m(\infty)=a\}} \right)=\lim_{u\to \infty}\E\left( f(  {\Pi_+^m}_{\vert_{[0,T]}})1_{\{\xi^m(u)=a \}}\right).$$
Let us fix $\varepsilon >0$. We choose $M\ge 0$ such that  if $\lambda\in P_+$ and satisfies $$\langle\alpha^\vee_i,\lambda\rangle\ge M,\quad \textrm{for } i\in\{0,1\},$$ then one has $$ a\in B(\lambda)\, \textrm{ and } \,\left\vert \frac{1}{e^{-\langle{\lambda,\rho^\vee/m\rangle}}\mbox{ch}_{\lambda}(\rho^\vee/m)}-\frac{1}{\mbox{ch}_{M(0)}(\rho^\vee/m)}\right\vert\le \varepsilon.$$
Lemma \ref{LGNC} implies that there exists $u_0\in \N$ such that for all integer $u\ge u_0$ 
$$\P\left( \langle\alpha^\vee_i,\Pi_+^m(u)\rangle\ge M, i\in\{0,1\}\right)\ge 1-\varepsilon.$$
By conditioning on $\{\Pi_+^m(t), 0\le t\le u\}$ in the lefthand side expectation of the following identity  one obtains by proposition \ref{preprop-string},   for  an integer $u\ge T$,
$$ \E\left( f(  {\Pi_+^m}_{\vert_{[0,T]}})1_{\{\xi^m(u)=a\} }\right)=\E\left(  f(  {\Pi_+^m}_{\vert_{[0,T]}})\frac{e^{-\langle \omega(a),\rho^\vee/m\rangle}1_{B(\Pi_+^m(u))}(a)}{e^{-\langle \Pi_+^m(u),\rho^\vee/m\rangle}\mbox{ch}_{\Pi_+^m(u)}(\rho^\vee/m)}\right).$$
  It implies that for an integer $u\ge u_0$,
$$\left\vert \E\left( f(  {\Pi_+^m}_{\vert_{[0,T]}})1_{\{\xi^m(u)=a\} }\right)-\E\left(  f(  {\Pi_+^m}_{\vert_{[0,T]}})\right)\frac{e^{-\langle \omega(a),\rho^\vee/m\rangle} }{\mbox{ch}_{M(0)}(\rho^\vee/m)}\right\vert \le 2\varepsilon,$$
which gives the lemma.
\end{proof}
Proposition \ref{prop-string} implies immediately the following corollary. 
\begin{coro} \label{Verm-Dem} For $p\ge 0$,
$$\P\left(\xi_{p+1}^m(\infty)=0\right)=\frac{\mbox{ch}^{w_p}_{M(0)}(\rho^\vee/m)}{\mbox{ch}_{M(0)}(\rho^\vee/m)}.$$
\end{coro} 
Since  $$\{\Pi^m(t)\in \mathcal C, \, t\ge 0 \}=\{\xi^m(\infty)=0\},$$  Proposition \ref{prop-string} has a  second corollary, which  has already been proved in \cite{LLP} by a quite different method. This corollary is not   useful for our purpose, nevertheless  it is worth giving it.
 \begin{coro}\label{LLPproba}   One has
 $\P\left(\Pi^m(t)\in \mathcal C, \, t\ge 0\right) =(\mbox{ch}_{M(0)}(\rho^\vee/m))^{-1}.
 $ \end{coro}
 
\section{The continuous counterpart}\label{The continuous counterpart}
 The  random processes introduced in section \ref{RWLP} are approximations of continuous time  random processes  defined in this section. For this, let us define the affine cone $$C_{\mbox{aff}}=\{(t,x)\in\R_+\times \R_+ : 0<x<t \}.$$ Let $\{B(t)=t\Lambda_0+(b_t+t/2)\alpha_1/2: t\ge 0\},$  where $\{b_t: t\ge 0\}$ is a standard real Brownian motion starting from $0$. Let $\varphi_{1/2}$ be a function defined on $\R_+^*\times \R$ by  
\begin{align}\label{defharm}\varphi_{1/2}(t,x)=
e^{- x/2}
\sum_{k\in \Z}{ \sinh((2kt+x)/2)}e^{-2(kx+k^2t)},  \textrm{ for } t>0, x\in \R.
\end{align}   
This is  an harmonic function for the process $B$  killed  on  the boundary of   $C_{\mbox{aff}}$. It is positive on $C_{\mbox{aff}}$ and vanishies on the boundary of $C_{\mbox{aff}}$.   Let $\{A(t), t \geq 0\}$ be the process starting from $(0,0)$, whose   law is    the Doob transformation of the law of  the process  $B$  killed on the boundary of   $C_{\mbox{aff}}$  by the   function $\varphi_{1/2}$.  This process has been introduced and studied in   \cite{ defo2, defo3} and carefully defined in \cite{boubou-defo} in the context of the present paper.
 
 The convergences in the following proposition have been proved in  \cite{defo4}. In this proposition, as in the convergence theorems of the following    sections, all the processes are considered as processes with values in     the quotient space $\mathfrak h^*_\R\mod \delta$, which is identified with $\R\Lambda_0\oplus \R\alpha_1=\R^2$.  We notice that $\alpha_0=-\alpha_1 $ in the quotient space. The set of continuous functions from $\R_+$ to $\R^2$  is equipped with the topology of uniform convergence on compact sets and we use the standard definition of convergence in distribution for a sequence of continuous processes as in Revuz and Yor (\cite{revuzyor}, XIII.1). 

\begin{prop} \label{conv-conj} 
\begin{enumerate}
\item For any $t\ge 0$, the random variable $(\Pi^m(mt)-\Pi^m\lfloor mt\rfloor)/m$ goes to $0$ in probability when $m$ goes to infinity.
\item The sequence  of processes
$$\{\frac{1}{m}\Pi^m(mt): t\ge 0\},  \quad m\ge 1,$$
 viewed in the  quotient space $\mathfrak h^*_\R \mod \delta$, converges in distribution towards the process $\{B(t) : t\ge 0\}$ when $m$ goes to infinity.  \item The sequence of processes  $$\{\frac{1}{m}  \Pi_+^{m}\lfloor mt\rfloor: t\ge 0\}, \quad m\ge 1,$$  viewed in the  quotient space $\mathfrak h^*_\R \mod \delta$, converges   towards $\{A(t): t\ge 0\}$ when $m$ goes to infinity,   in the sense of finite dimensional distributions.
 \end{enumerate}

\end{prop} 
For $t\ge 0$, we consider the string coordinates  of $B$ on $[0,t]$, denoted by $(\xi_k(t))_{k\ge 0}$. They are defined by 
\begin{align}\label{def-string-mb} 
\mathcal P_{\alpha_{m}}\dots\mathcal P_{\alpha_{0}}B(t)  =B(t)+\sum_{k=0}^m\xi_k(t)\alpha_{k},\quad m\ge 0.
\end{align} 

For every $k\ge 0$, the  function $t\in \R_+ \mapsto \xi_k(t)$ is increasing, and because of the drift, $\lim_{t\to \infty} \xi_k(t)<+\infty$. We set  $\xi_k(\infty)=\lim_{t\to \infty }\xi_k(t)$. 
 For a sequence $x=(x_k)\in \R_+^{\N}$, we set
\begin{align}\label{sigma}\omega(x)=\lim_{n \to +\infty} \sum_{k=0}^{n-1}x_k\alpha_k+\frac{1}{2}x_n\alpha_n \mod\delta,\end{align}
  when this limit exists in $\R\alpha_1$.  
The following sets are the continuous analogs of the Kashiwara crystals  defined in definition \ref{crystals}.

  \begin{definition} \label{ccrystals} One defines,  for $\lambda\in \bar{ C}_{\mbox{aff}} $, 
\begin{align*}&\Gamma(\infty)=\{x=(x_k)\in  \R_+^\N : \frac{x_k}{k}\ge \frac{x_{k+1}}{k+1}\ge 0,  \textrm{ for all } k\ge 1, \,  \omega(x) \in\R^2\},\\
&\Gamma(\lambda)=\{x\in \Gamma(\infty) : x_k\le \langle\lambda-\omega(x)+\sum_{i=0}^kx_i\alpha_i,\alpha^\vee_k\rangle, \textrm{ for every }  k\ge 0\}.
\end{align*}
\end{definition}

\section{An inverse Pitman's theorem} \label{An inverse Pitman's theorem}
We will now prove a reconstruction theorem which allows to get a space-time Brownian motion $B$ from a  conditioned one $A$ and a sequence of random variables properly distributed. The idea is to prove that the commutative diagram in figure \ref{BLP2} is valid. The convergence represented by the third arrow of the diagram  will then provide a reconstruction theorem. Black arrows on the diagram stand for convergences that have been already proved. Dashed ones stand for convergences   which have still to be proved at this stage. Let us first define the random variables involved in the diagram which have not been defined yet. The law of $\xi(\infty)$ is described by 
the following theorem, which has been proved in \cite{boubou-defo}.

 \begin{theorem}[Ph. Bougerol, M. Defosseux \cite{boubou-defo}] \label{loidexi}  The random variables
 $$\xi_0(\infty) ,  \quad \frac{1}{2}((k+1)\xi_{k}(\infty) -k\xi_{k+1}(\infty)), \quad k\ge 1,$$
are independent   exponential random variables with parameter $1$.
\end{theorem}
   \begin{figure}[h] 
     \begin{center}
          \includegraphics[scale=0.8]{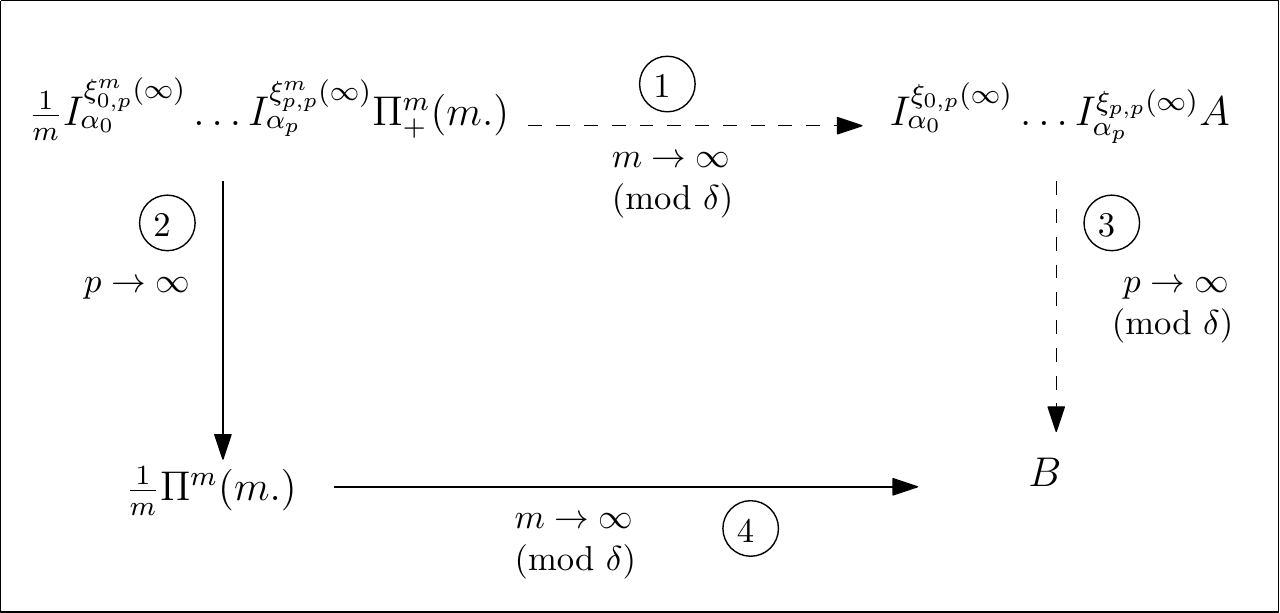}  
          
            \caption{A commutative diagram of  finite dimensional distributions convergences}
            \label{BLP2}
                    \end{center} 
 \end{figure}
 
From now on,  $\varepsilon_n, n \geq0$, is a sequence of independent exponential random variables with parameter $1$ defined by 
 \begin{align}\label{expo}
 \varepsilon_0=\xi_0(\infty) ,  \quad \varepsilon_k= \frac{1}{2}((k+1)\xi_{k}(\infty) -k\xi_{k+1}(\infty)), \quad k\ge 1,
 \end{align}
and  $\{A(t): t\ge 0\}$ is supposed to be independent of this sequence.
\begin{definition} \label{xi-cond} For every $p\ge 0$,   let $\xi_{0,p}(\infty)=\varepsilon_0,$ and let $\xi_{k,p}(\infty)$ be defined by
$$\frac{\xi_{k,p}(\infty)}{k}=\sum_{n=k}^{p} \frac{ 2\varepsilon_n}{n(n+1)},$$
for all $k \in\{1,\dots,p\}$. 
We write $\xi_{\cdot,p}(\infty)=(\xi_{k,p}(\infty))_{k\in \{0,\dots,p\}}$.
\end{definition}

\subsection{Proof of the convergence corresponding to the first arrow of the diagram}
 For every $p\ge 0$, let 
 $$(\xi_{0,p}^{m}(\infty), \dots,\xi_{p,p}^{m}(\infty))$$ be a random vector  independent from $\Pi_+^m$, which is distributed as  $(\xi_0^m(\infty), \dots \xi_p^m(\infty))$ conditionally on  $\xi_{p+1}^m(\infty)=0$.  Lemma \ref{conv-xi-cond} and Propositions \ref{ConvPiPlus} and \ref{InfpourPpi} will imply the  desired convergence. 
 \begin{lemma} \label{conv-xi-cond} For every $p\in \N$, $\frac{1}{m}(\xi_{0,p}^{m}(\infty), \dots,\xi_{p,p}^{m}(\infty))$ converges in distribution  towards $(\xi_{0,p}(\infty),\dots,\xi_{k,p}(\infty))$ when $m$ goes to $+\infty$.  
\end{lemma} 
\begin{proof} From definition \ref{xi-cond}, one derives  that  the density  of  $(\xi_{0,p}(\infty), \dots, \xi_{p,p}(\infty))$  is  given by
$$f_{(\xi_{0,p},\dots,\xi_{0,p})} (x_0, \dots, x_p) = \dfrac{(p+1)! e^{- \sum_{k=0}^p x_k}}{2^p} 1_{x_0 \geq 0, \frac{x_1}{1} \geq \frac{x_2}{2} \geq \cdots \geq \frac{x_p}{p} \geq 0}.$$
Moreover, from   Proposition \ref{prop-string} and Corollary \ref{Verm-Dem} we deduce  that for every real numbers $t_0,\dots,t_p\ge 0$,
 
$$\E \left( e^{- \sum_{k=0}^p t_k \frac{\xi_{k,p}^{m}(\infty)}{m}} \right)=  \frac{1}{\mbox{ch}^{w_p}_{M(0)}(\rho^\vee/m)} \sum_{(a_0, \dots, a_p) \in \mathbb{N}^{p+1}} e^{- \sum_{k=0}^p (1+t_k)\frac{a_k}{m}} 1_{ \frac{a_1}{1} \geq \frac{a_2}{2} \geq \cdots \geq \frac{a_p}{p}}.$$
Lemma follows from the fact that
$$m^{-(p+1)}\sum_{(a_0, \dots, a_p) \in \mathbb{N}^{p+1}} e^{- \sum_{k=0}^p (1+t_k)\frac{a_k}{m}} 1_{ \frac{a_1}{1} \geq \frac{a_2}{2} \geq \cdots \geq \frac{a_p}{p}}$$
converges towards the Riemann integral
$$ \int_{\R_+^{p+1}}e^{- \sum_{k=0}^p (1+t_k)x_k} 1_{ \frac{x_1}{1} \geq \frac{x_2}{2} \geq \cdots \geq \frac{x_p}{p}}\, dx.$$ 
\end{proof}

\begin{prop}\label{cvprobaP} For every $t\ge 0$, $\frac{1}{m}(  \Pi_+^m(mt)- \Pi_+^m\lfloor mt \rfloor)$
converges in probabilty to $0$ as $m$ goes to infinity. 
\end{prop}
\begin{proof}
Let us fix $\varepsilon>0$ and $t>0$. We choose a compact $K$ in $C_{\mbox{aff}}$ such that 
$$\P(A(t)\in K)>1-\varepsilon/2.$$ Convergences recalled in proposition \ref{conv-conj} ensure that there exists $m_0\in \N^*$ such that for all $m\ge m_0$ 
$$\P\left(\frac{1}{m}(\Pi_+^m\lfloor mt\rfloor+\rho)\in K\right)>1-\varepsilon.$$
We choose such an integer $m_0$. One has for all $ m \geq m_0$
\begin{align}\label{refer} \E&\left( 1_{\{\frac{1}{m} \vert \langle \Pi_+^m(mt)-\Pi_+^m\lfloor mt \rfloor, \alpha_1^\vee\rangle\vert > \varepsilon\} }\right)\nonumber \\ &\quad \quad\quad\quad\quad\quad \leq\E\left( 1_{\{ \frac{1}{m} \vert \langle \Pi_+^m(mt)-\Pi_+^m\lfloor mt \rfloor, \alpha_1^\vee\rangle\vert >\varepsilon\}\cap K_m}\right) + \varepsilon,
\end{align}
where $K_m=\{ \frac{1}{m}(\Pi_+^m\lfloor mt \rfloor+\rho)\in K \}$.
By proposition \ref{fromMtoC},  one has for $\lambda\in P_+$,
\begin{align*}
 &\E \left( 1_{\{\frac{1}{m}\vert \langle \Pi_+^m(mt)-\Pi_+^m\lfloor mt \rfloor, \alpha_1^\vee\rangle\vert > \varepsilon\}} | \Pi_+^m\lfloor mt \rfloor =\lambda \right) \\
    &= \E \left( 1_{ \{\frac{1}{m} \vert \langle \Pi^m(mt -\lfloor mt \rfloor), \alpha_1^\vee\rangle\vert > \varepsilon\}} \frac{\mbox{ch}_{\Pi^m(1)+\lambda}(\rho^\vee/m)}{\mbox{ch}_{\lambda}(\rho^\vee/m)}e^{-\langle\Pi^m(1),\rho^\vee/m\rangle}1_{\{\lambda+\Pi^m_{\vert [0,1]} \in C_{\mbox{aff}}\}} \right)	 \\
    &\le \E \left( 1_{\{ \frac{1}{m} \vert \langle \Pi^m(mt-\lfloor mt \rfloor), \alpha_1^\vee\rangle\vert > \varepsilon\}} \frac{\mbox{ch}_{M(0)}(\rho^\vee/m)}{\mbox{ch}_{\lambda}(\rho^\vee/m) e^{-\langle\lambda,\rho^\vee/m\rangle} }  \right),
\end{align*} 
the last inequality being derived from (\ref{VermaMaj}).
Moreover the Weyl character formula gives
$$ \frac{\mbox{ch}_{\lambda}(\rho^\vee/m) e^{-\langle\lambda,\rho^\vee/m\rangle} }{\mbox{ch}_{M(0)}(\rho^\vee/m)} = \sum_{w \in W} \det(w) e^{\langle w(\frac{\lambda+ \rho^\vee}{m} ) - (\frac{\lambda+\rho^\vee}{m} ),\rho^\vee\rangle}.$$
The function $$x\in C_{\mbox{aff}}\mapsto  \sum_{w \in W} \det(w) e^{\langle w(x) - x,\rho^\vee\rangle},$$ is   positive on $K$. We set
$$M=\max\{(\sum_{w \in W} \det(w) e^{\langle w(x) - x,\rho^\vee\rangle})^{-1}: x\in K\}.$$
Thus for $\lambda\in P_+$ such that $(\lambda+\rho)/m\in K$ one has,
$$  \E \left( 1_{\{\frac{1}{m} \vert \langle \Pi_+^m(mt)-\Pi_+^m\lfloor mt \rfloor, \alpha_1^\vee\rangle\vert > \varepsilon\}} \vert   \Pi_+^m\lfloor mt \rfloor =\lambda \right) \le M\E  \left( 1_{ \{\frac{1}{m} \vert \langle \Pi^m (mt - \lfloor mt \rfloor) , \alpha_1^\vee\rangle \vert > \varepsilon\}}\right).$$ 
As $ \frac{1}{m} \Pi^m (mt - \lfloor mt \rfloor )$ converges towards $0$ in probability as it is recalled in proposition \ref{conv-conj}, we choose an integer $m_1\ge m_0$ such that for all $m\ge m_1$,
$$\E  \left( 1_{\{ \frac{1}{m} \vert \langle \Pi^m (mt - \lfloor mt \rfloor),\alpha_1^\vee\rangle \vert > \varepsilon\}}\right)\le \varepsilon/M.$$
Finally by conditioning by $\Pi_+^m\lfloor mt\rfloor$ within the expectation of the righthand  side of inequality (\ref{refer}), one obtains   for $m\ge m_1$, 
\begin{align} \E  \left( 1_{\{\frac{1}{m} \vert \langle \Pi_+^m(mt)-\Pi_+^m\lfloor mt \rfloor, \alpha_1^\vee\rangle\vert > \varepsilon\}}\right) \leq2 \varepsilon,
\end{align}
which proves the expected convergence.
\end{proof}
By proposition \ref{fromMtoC}, we prove in the following proposition that the   sequence of random processes $\{\frac{1}{m} \Pi_+^m(mt):t\ge 0\},$ $ m\ge 1,$ inherits the tightness from  $\{\frac{1}{m} \Pi^m(mt):t\ge 0\},$ $ m\ge 1$.
 
\begin{prop}\label{tension} The sequence of processes $\{\frac{1}{m} \Pi_+^m(mt):t\ge 0\},$ $ m\ge 1,$ is tight.
\end{prop}
\begin{proof}
For $t\ge 0$, we set $X^m (t) = \frac{1}{m} \Pi_+^m(mt)$.  As it has been recalled in proposition \ref{conv-conj}, it has been proved in \cite{defo4} that  $\frac{1}{m}\Pi_+^m\lfloor mt \rfloor$ converges in law when $m$ goes to  infinity. From proposition \ref{cvprobaP}, we deduce     the convergence in law of $X^m (t)$ for any $t\ge 0$. Thus it is sufficient to prove  that
$$\forall T \geq 0, \forall \varepsilon >0, \forall \eta >0, \exists \delta >0 \ \text{s.t.} \ \limsup_{m \to +\infty} \P \left( w_T (X^m, \delta) \geq \eta \right) \leq \varepsilon,$$
where, for    $x: \R_+\to \mathfrak h_\R^*$, $$w_T (x,h) = \sup\{ \vert \langle x(t)-x(s),\alpha_1^\vee\rangle\vert, s,t \in [0,T], \vert s-t\vert \leq \delta\}.$$
Let $T,\varepsilon,\eta>0$. We suppose that $T$ is greater than $\eta$. We set $t_0 = \frac{\eta}{2} $ and define $w_T ^{t_0} (x, \delta) $ by
 $$w_T ^{t_0} (x, \delta) = \sup\{ \vert \langle x(t)-x(s),\alpha_1^\vee\rangle\vert, s,t \in [ t_0 , T], |s-t| \leq \delta\},$$
 for  $\delta\ge 0$, $x: \R_+\to \mathfrak h_\R^*$,
 As  for every $t\ge 0$, $X^m(t)$ is   in    $C_{\mbox{aff}}$, one has for  $ \delta\le t_0$,
$$ \left\{ w_T (X^m , \delta) \geq \eta \right\} \subset \left\{ w_T ^{t_0} (X^m, \delta) \geq \eta \right\}.$$ 
As in the proof of proposition \ref{cvprobaP} we choose  a compact $K$ in $C_{\mbox{aff}}$ and $m_0 \in \mathbb{N}$  such as for all $ m \geq m_0$ 
$$\P \left( \frac{\Pi_+^m\lfloor mt_0 \rfloor+\rho}{m} \in K \right) \geq 1 - \varepsilon.$$
Hence,
\begin{align*}
 \mathbb{P}\left(  w_T ^{t_0} (X^m, \delta) \geq \eta  \right) &\leq \E\left(1_{\{w_T ^{t_0} (X^m, \delta) \geq \eta, \frac{ \Pi_+^m\lfloor mt_0 \rfloor+\rho}{m} \in K \}} \right)+ \varepsilon.
\end{align*}
By conditioning by $ \Pi_+^m\lfloor mt_0\rfloor$ in the expectation of the righthand side of the above inequality, we obtain as in the proof of the proposition \ref{cvprobaP} that   
$$\P\left(w_T ^{t_0} (X^m, \delta) \geq \eta, \frac{\Pi_+^m\lfloor mt_0 \rfloor+\rho}{m} \in K  \right)\le M\P\left(w_T  (\frac{1}{m}\Pi^m(m.), \delta) \geq \eta\right),$$
where $$M=\max\{(\sum_{w \in W} \det(w) e^{(w(x) - x),\rho^\vee})^{-1}: x\in K\}.$$ As the sequence of processes $\{\frac{1}{m}\Pi^m(mt), t\ge 0\}$, $m\ge 0$, is tight, we choose $m_1\ge m_0$ and $\delta_0 \in(0, \eta/2]$ such that for $m\ge m_1$,
$$\P\left(w_T  (\frac{1}{m}\Pi^m(m.), \delta_0) \geq \eta\right)\le \varepsilon/M.$$
Thus  for $m\ge m_1$, one has  
\begin{align*}
 \mathbb{P}\left(  w_T ^{t_0} (X^m, \delta_0) \geq \eta  \right) &\leq  2\varepsilon,  
\end{align*}
which ends the proof.
\end{proof}
The convergence recalled in proposition \ref{conv-conj} of $\{\frac{1}{m}\Pi_+^m\lfloor mt\rfloor:t\ge 0\}$ in the sense of finite dimensional law  and the previous proposition give the following one.
\begin{prop}\label{ConvPiPlus}  The sequence of processes $\{\frac{1}{m}  \Pi_+^m(mt):t\ge 0\},$ $ m\ge 1,$ converges in distribution towards $\{A(t): t\ge 0\}$ in the quotient space $\mathfrak h^*_\R \mod \delta$.
\end{prop}
Now it remains to control the asymptotic behavior of $\Pi^m_+$ for large time  uniformly in $m$ in order to get the convergence represented by the first arrow in the diagram. For this we show the following proposition.
\begin{prop}\label{InfpourPpi} For all $\varepsilon,a >0$ there exists $T,m_0\ge 0$ such as for $i\in \{0,1\}$ and all $m\ge m_0$ $$\min \left( \P\left(\inf_{t\ge T}\frac{1}{m}\langle \Pi^m(mt),\alpha_i^\vee\rangle \ge a\right),\, \, \P\left(\inf_{t\ge T}\frac{1}{m}\langle\Pi_+^m(mt),\alpha_i^\vee\rangle \ge a\right) \right) \ge 1-\varepsilon.$$
\end{prop}
\begin{proof}  Slight modifications in the proof of proposition 6.13 of \cite{defo4} give the first inequality. Let $i\in \{0,1\}$. 
As previously we choose  a compact $K$ in $  C_{\mbox{aff}}$  and $m_0 \in \mathbb{N}^*$   such as for all $  m \geq m_0$
$$\P \left( \frac{ \Pi_+^m(m)+\rho}{m} \in K \right) \geq 1 - \varepsilon .$$ 
Let $T \geq 1$    that will be chosen later. For $u > T$ and $  m \geq m_0$ one has
\begin{align*}
 &\E(1_{\{\inf\{\frac{1}{m}\langle \Pi_+^m(mt),\alpha_i^\vee\rangle,  T \leq t \leq u\} \le a\}})  \leq \E \left(1_{ \{\inf\{\frac{1}{m}\langle \Pi_+^m(t),\alpha_i^\vee\rangle, \lfloor mT \rfloor \leq t \leq \lfloor mu \rfloor + 1\} \le a\}\cap   K_m} \right) + \varepsilon,
\end{align*}
where $K_m=\{ \frac{1}{m}(\Pi_+^m\lfloor mt \rfloor+\rho)\in K \}$.
 By conditioning by $ \Pi_+^m\lfloor m\rfloor$ in the expectation of the righthand side of the above inequality, we obtain as in the proof of   Proposition \ref{cvprobaP} that there exists $M\ge 0$ such that 
\begin{align*}
 \E \left(1_{\{ \inf\{\frac{1}{m}\langle \Pi_+^m(t),\alpha_i^\vee\rangle, \lfloor mT \rfloor \leq t \leq \lfloor mu \rfloor + 1\} \le a\}\cap K_m} \right)  \leq  M \E \left( 1_{\{\inf\{\frac{1}{m}\langle \Pi^m(t),\alpha_i^\vee\rangle, \lfloor mT \rfloor - m \leq t\} \le a \}}\right).
\end{align*}
Thanks to the first inequality, for such an $M\ge 0$, we choose  $T_0 \geq 0$ and $m_1 \ge m_0$ such as for   $m \geq m_1$
$$\P \left( \inf_{ \lfloor mT_0\rfloor-m \leq t}\frac{1}{m}\langle \Pi^m(t),\alpha_i^\vee\rangle \le a \right) \leq \varepsilon/M.$$
Thus for $u\ge T_0$, $m \geq m_1$
$$\P\left(\inf_{T_0\le t\le u}\frac{1}{m}\langle  \Pi_+^m(mt),\alpha_i^\vee\rangle \le a\right)\le 2\varepsilon.$$
As $m_1$ does not depend on $u$, we let $u$ goes to infinity in the above inequality, which ends the proof.

\end{proof}

We can now state the convergence corresponding to the first arrow of the diagram of figure \ref{BLP2}.
\begin{prop}  In the quotient space $\mathfrak h^*_\R\mod \delta$, the sequence of random processes $$\{\frac{1}{m}I_{\alpha_0}^{\xi_{0,p}^{m}(\infty)}\dots I_{\alpha_p}^{\xi_{p,p}^{m}(\infty)}\Pi_+^{m}(mt):t\ge 0\}, \quad m\ge 1,$$
converges in  distribution towards
$$\{I_{\alpha_0}^{\xi_{0,p}(\infty)}\dots I_{\alpha_p}^{\xi_{p,p}(\infty)}A(t), t\ge 0\},$$
as $m$ goes to infinity.
\end{prop}
 
\subsection{Proof of the convergence corresponding to the third  arrow of the diagram}  Let us  first notice  that  Proposition \ref{prop-string} implies the following one.
\begin{prop} \label{partial-rec} For $m,p\ge 1$, the process 
$ \{I_{\alpha_0}^{\xi_{0,p}^{m}(\infty)}\dots I_{\alpha_p}^{\xi_{p,p}^{m}(\infty)}  \Pi_+^{m}(t): t\ge 0\}$
has the same law as $\{\Pi^m(t):t\ge 0\}$ conditionally on $\{\xi_{p+1}^m(\infty)=0\}$.
\end{prop}

\begin{prop} For $u\in \R$, and $p\in \N$,
\begin{align}\label{first}
\E\left(e^{iu\langle\Pi^m\lfloor mt \rfloor/m,\alpha_1^\vee\rangle}1_{\xi_{p+1}^m \lfloor mt \rfloor=0}\right)=\E \left( \frac{\mbox{ch}^{w_p}_{ \Pi_+^m \lfloor mt \rfloor}(\frac{1}{m}(iu\alpha^\vee_1+\rho^\vee))}{\mbox{ch}_{ \Pi_+^m \lfloor mt \rfloor}(\frac{1}{m}\rho^\vee)} \right).
\end{align}
In particular,
\begin{align}\label{second}
\P\left(\xi_{p+1}^m \lfloor mt \rfloor=0\right)=\E \left( \frac{\mbox{ch}^{w_p}_{ \Pi_+^m \lfloor mt \rfloor}(\frac{1}{m}\rho^\vee)}{\mbox{ch}_{ \Pi_+^m \lfloor mt \rfloor}(\frac{1}{m}\rho^\vee)} \right).
\end{align}
\end{prop}
\begin{proof} First notice that  Identity (\ref{second}) follows by letting $u=0$ in (\ref{first}). To prove (\ref{first}), we notice that proposition \ref{preprop-string} implies that 
\begin{align*}
\E(e^{iu\langle\Pi^m\lfloor mt \rfloor/m,\alpha_1^\vee\rangle}1_{\xi_{p+1}^m \lfloor mt \rfloor=0}\vert   \Pi_+^m\lfloor mt\rfloor=\lambda)\end{align*} is equal to $$\frac{\sum_{a\in B(\lambda)}e^{iu\langle(\lambda-\omega(a))/m,\alpha_1^\vee\rangle}e^{\langle\lambda-\omega(a),\rho^\vee/m\rangle}1_{a_{p+1}=0}}{\mbox{ch}_\lambda(\rho^\vee/m)}$$ which is by (\ref{CDString}) equal to $$ \frac{\mbox{ch}_\lambda^{w_p}((iu\alpha_1^\vee+\rho^\vee)/m)}{\mbox{ch}_\lambda(\rho^\vee/m)}. $$
Thus (\ref{first})   follows by conditioning by $ \Pi_+^m\lfloor mt\rfloor$ within the lefthand side expectation of the identity.
 
\end{proof}
 The idea of the proof of the third convergence of the diagram rests on the fact that 
$$\E\left(e^{iu\langle\Pi^m\lfloor mt\rfloor/m,\alpha_1^\vee\rangle}\vert \xi_{p+1}^m\lfloor mt\rfloor=0\right)$$ 
for which an explicit formula involving a Demazure character   is available as we have just seen, is 
close to 
$$\E\left(e^{iu\langle\Pi^m\lfloor mt\rfloor/m,\alpha_1^\vee\rangle}\vert \xi_{p+1}^m(\infty)=0\right)$$
whose limit we are looking for. 

\begin{definition} Let  $(L_p)_{p\ge 0}$ be  the random sequence defined by 
$$L_p=\sum_{k=0}^p \xi_{k,p}(\infty)\alpha_k, \quad p\ge 0.$$
\end{definition} 
 Lemma \ref{conv-xi-cond} implies in particular that for any $p\ge 0$, the random variable $\frac{1}{m}\omega(\xi^m_{.,p}(\infty))$ converges in distribution towards $L_p$ when $m$ goes to infinity. Notice that viewed in $\mathfrak h_\R^*\mod \delta $,  $(L_p)_{p\ge 0}$  is a sequence of real numbers. 
\begin{lemma} \label{LimDesLp} In $\mathfrak h_\R^*\mod \delta $,  $L_p$ converges almost surely and in $L^2$  towards 
$$L =\sum_{k=0}^\infty \frac{\varepsilon_k}{2 \lfloor k/2 \rfloor +1}\alpha_k \mod \delta,$$
when $p$ goes to infinity.
\end{lemma}
\begin{proof}
One has for $p\ge 0$, 
$$L_p =\varepsilon_0\alpha_0+ \sum_{k=1}^p k \sum_{n=k}^p \frac{2 \varepsilon_n}{n(n+1)} \alpha_k = \varepsilon_0\alpha_0+\sum_{n=1}^p \frac{2 \varepsilon_n}{n(n+1)} \sum_{k=1}^n k \alpha_k.$$
Thus
\begin{align*}\langle L_p, \alpha^\vee_1\rangle &= -2+\sum_{n=1}^p \frac{4 \varepsilon_n}{n(n+1)} \sum_{k=1}^n k (-1)^{k+1}\\
&=  -2+\sum_{n=1}^p \frac{4 \varepsilon_n}{n(n+1)} (-1)^{n+1}\lfloor\frac{n+1}{2}\rfloor.
\end{align*}
 Finally $$L_p = \sum_{k=0}^p \frac{\varepsilon_k}{2 \lfloor k/2 \rfloor +1}\alpha_k \mod \delta,$$
which shows   that in $\mathfrak h_\R^*\mod \delta $, $(L_p)_{p\ge 0}$ is a bounded martingale  in $L^2$, and gives the expected convergence. 
\end{proof}

\begin{lemma}\label{convD} If $(\lambda_m)$ is a sequence with values in $P_+$ such that $\lim_{m\to\infty}\frac{\lambda_m}{m}=\lambda\in   C_{\mbox{aff}}$ then  for $u\in \R$,
$$\lim_{m\to \infty}\frac{\mbox{ch}_{\lambda_m}^{w_p}((iu\alpha^\vee_1+\rho^\vee)/m)}{\mbox{ch}^{w_p}_{M(0)}(\rho^\vee/m)}=e^{\langle \lambda,\rho^\vee\rangle}\E\left(e^{iu\langle \lambda-L_p,\alpha^\vee_1\rangle}1_{\xi_{\cdot,p}\in \Gamma(\lambda)}\right)$$
\end{lemma}
\begin{proof} Expressions (\ref{CDString}) and (\ref{VDString}) for Demazure characters give 
$$\frac{1}{m^{p+1}}\mbox{ch}_{\lambda_m}^{w_p}((iu\alpha_1^\vee+\rho^\vee)/m)=\frac{1}{m^{p+1}}\sum_{a\in B(\lambda_m), a_{p+1}=0}e^{\langle\frac{1}{m}(\lambda_m-\omega(a)),iu\alpha_1^\vee+\rho^\vee\rangle},$$  and 
$$\frac{1}{m^{p+1}}\mbox{ch}^{w_p}_{M(0)}(\rho^\vee/m)=\frac{1}{m^{p+1}}\sum_{a\in B(\infty),  a_{p+1}=0}e^{ -\langle\frac{1}{m}\omega(a), \rho^\vee\rangle}. $$   
Thus 
\begin{align*}
\lim_{m\to \infty}\frac{\mbox{ch}_{\lambda_m}^{w_p}((iu\alpha^\vee_1+\rho^\vee)/m)}{\mbox{ch}^{w_p}_{M(0)}(\rho^\vee/m)}
&=\frac{\int_{\R_+^{p+1}}e^{\langle\lambda-\omega(x), iu\alpha_1^\vee+\rho^\vee\rangle}1_{x\in \Gamma(\lambda)}\, dx}{\int_{\R_+^{p+1}}e^{-\langle \omega(x), \rho^\vee\rangle}1_{x\in \Gamma(\infty)}\, dx}\\
&=e^{\langle\lambda,  \rho^\vee\rangle}\frac{\int_{\R_+^{p+1}}e^{iu\langle\lambda-\omega(x), \alpha_1^\vee\rangle}e^{-\langle \omega(x),  \rho^\vee\rangle}1_{x\in \Gamma(\lambda)}\, dx}{\int_{\R_+^{p+1}} e^{-\langle \omega(x),\rho^\vee\rangle}1_{x\in \Gamma(\infty)}\, dx}.
\end{align*}
The observation of the density of $(\xi_{0,p}(\infty),\dots,\xi_{p,p}(\infty))$ given in the proof of Lemma \ref{conv-xi-cond} allows to conclude.
\end{proof}

\begin{lemma} \label{conv-cryst} For  $\lambda\in   C_{\mbox{aff}}$,  the random variable $1_{\{\xi_{\cdot,p}(\infty)\in \Gamma(\lambda)\}}$ converges almost surely towards $1_{\{\xi(\infty)\in \Gamma(\lambda)\}}$ when $p$ goes to infinity.
\end{lemma}
\begin{proof} We know that almost surely in the quotient space $\mathfrak h_\R^*\mod \delta$ 
 $$\lim_{k\to\infty}\sum_{i=0}^{k-1}\xi_i(\infty)\alpha_i+\frac{1}{2}\xi_k(\infty)\alpha_k=L  ,\textrm{ and } \lim_{p\to \infty}L_p=L.$$ 
 As for every integer $k$, almost surely $\lim_{p\to\infty}\xi_{k,p}(\infty)=\xi_{k}(\infty)$, one obtains a first   inclusion
 \begin{align*}\limsup_{p\to \infty}\{\xi_{.,p}(\infty)\in \Gamma(\lambda)\}\subset  \{\xi(\infty)\in \Gamma(\lambda)\}.\end{align*}
  We set for $k\in \{1,\dots,p\}$
$$X_{k,p}=L_p-\sum_{i=0}^{k-1}\xi_{i,p}(\infty) \alpha_i-\frac{1}{2}\xi_{k,p}(\infty)\alpha_{k} $$ and $$X_k=L-\sum_{i=0}^{k-1}\xi_i (\infty)\alpha_i-\frac{1}{2}\xi_{k}(\infty)\alpha_{k} .$$
We notice that, in the quotient space, one has for   $k\in\{1,\dots, p\}$
\begin{align*} 
\sum_{i=0}^{k-1}\xi_{i,p} (\infty)\alpha_i+\frac{1}{2}\xi_{k,p}(\infty)\alpha_{k}=\sum_{i=0}^{k-1}\xi_i \alpha_i+\frac{1}{2}\xi_{k}(\infty)\alpha_{k}-\frac{\xi_{p}(\infty)}{2p}\alpha_k1_{k \textrm{ is odd}}.
\end{align*}
Thus almost surely 
$$\lim_{p\to \infty} \sup_{0\le k\le p} \vert \alpha_k^\vee(X_k^p-X_k)\vert=0.$$
It follows, as in the proof of proposition 5.14 of \cite{boubou-defo},    that almost surely
\begin{align*}   \{\xi(\infty)\in \Gamma(\lambda)\}\subset \liminf_{p\to\infty} \{\xi_{.,p}(\infty)\in \Gamma(\lambda)\}.\end{align*}
Finally one has,
\begin{align*}\limsup_{p\to \infty}\{\xi_{.,p}(\infty)\in \Gamma(\lambda)\}\subset  \{\xi(\infty)\in \Gamma(\lambda)\}\subset \liminf_{p\to\infty} \{\xi_{.,p}(\infty)\in \Gamma(\lambda)\},\end{align*}
from which the lemma follows. 
\end{proof} 
 
The function $\varphi_{iu+\frac{1}{2}}$ is defined on $\R_+^*\times \R$ by 
\begin{align}\label{defphi}\varphi_{iu+1/2}(t,x)=
\frac{e^{-(iu +1/2)x}}{\cosh(u)}
\sum_{k\in \Z}{ \sinh((iu+1/2)(2kt+x))}e^{-2(kx+k^2t)}, 
\end{align}   
$ t>0, x\in \R.$ 
 
\begin{prop} \label{string-cond} Let $\lambda\in C_{\mbox{aff}}$, $u\in \R$. One has 
$$\mathbb E(e^{-i  u\langle  L,\alpha_1^\vee\rangle}\vert \xi(\infty)\in \Gamma(\lambda))= \frac{\varphi_{iu+1/2}(\lambda)}{\varphi_{1/2}(\lambda)},\, \textrm{ and } \quad\P(\xi(\infty)\in \Gamma(\lambda))=2\varphi_{\frac{1}{2}}(\lambda).$$
\end{prop}
\begin{proof}  Notice that $L$ is the random variable denoted by $L^{(\mu)}(\infty)$ with $\mu=1/2$ in \cite{boubou-defo}.
The first identity follows   from Theorem 8.3 and proposition 6.7 of \cite{boubou-defo}. For the second one, we deduce from Theorems 5.2 and 5.5 of  \cite{bbobis} a similar  identity for the dihedral string coordinates defined in \cite{boubou-defo}. Then we apply proposition $5.14$ of \cite{boubou-defo}. 
\end{proof} Theorems  8.3 and  6.6 of \cite{boubou-defo} imply in particular the following proposition.
\begin{prop} \label{EcondB} For $u\in \R$,
$$\E\left(e^{iu\langle B(t),\alpha^\vee_1\rangle }\right)=\E\left(e^{iu\langle A(t),\alpha^\vee_1\rangle}\frac{\varphi_{iu+1/2}(A(t))}{\varphi_{1/2}(A(t))}\right).$$
\end{prop}
We have now all the ingredients needed to prove that the third convergence of the diagram  is valid, which   implies  Theorem \ref{reconstruction}.
\begin{theorem}
The sequence of processes $$\{I_{\alpha_0}^{\xi_{0,p}(\infty)}\dots I_{\alpha_p}^{\xi_{p,p}(\infty)}A(t), t\ge 0\}, \, p\ge 0,$$
converges when $p$ goes to infinity,  in a sense of finite dimensional distributions,  towards the space-time Brownian motion $\{B(t),t\ge 0\}$, in the quotient space $\mathfrak h_\R^*\mod \delta$.
\end{theorem}
\begin{proof} 
We  first prove    the convergence of $I_{\alpha_0}^{\xi_{0,p}(\infty)}\dots I_{\alpha_p}^{\xi_{p,p}(\infty)}A(t)$ for a fixed $t\ge 0$.  Let $t\ge 0$.
For $u\in \R$, $m,p\ge 1$, the Fourier transform
$$\E\left(e^{iu\langle I_{\alpha_0}^{\xi_{0,p}(\infty)}\dots I_{\alpha_p}^{\xi_{p,p}(\infty)}A(t),\alpha^\vee_1\rangle}\right)$$
is equal to 
$$\lim_{m\to \infty} \E\left(e^{i\frac{u}{m}\langle I_{\alpha_0}^{\xi^m_{0,p}(\infty)}\dots I_{\alpha_p}^{\xi^m_{p,p}(\infty)}\Pi^m_+\lfloor mt\rfloor,\alpha^\vee_1\rangle}\right),$$
which is, by Proposition \ref{partial-rec}, also equal to 
$$\lim_{m\to \infty} \E\left(e^{iu\langle\Pi^m\lfloor mt\rfloor/m,\alpha_1^\vee\rangle}\vert \xi_{p+1}^m(\infty)=0\right).$$
We write
$$\E\left(e^{iu\langle\Pi^m\lfloor mt\rfloor/m,\alpha_1^\vee\rangle}\vert \xi_{p+1}^m(\infty)=0\right)=S_1(u,m,p)+S_2(m,p)$$
where
\begin{align*}
S_1(u,m,p)&=\E\left(e^{iu\langle\frac{1}{m}\Pi^m\lfloor mt\rfloor,\alpha^\vee_1\rangle}1_{\{\xi_{p+1}^m\lfloor mt\rfloor=0\}}\right)/\P\left(\xi_{p+1}^m(\infty)=0\right)\\
&=\E\left(\frac{\mbox{ch}^{w_p}_{\Pi_+^m\lfloor mt\rfloor }(\frac{1}{m}(iu\alpha^\vee_1+\rho^\vee))}{\mbox{ch}_{\Pi_+^m\lfloor mt\rfloor}(\frac{1}{m}\rho^\vee)}\right)\frac{\mbox{ch}_{M(0)}(\rho^\vee/m)}{\mbox{ch}_{M(0)}^{w_p}(\rho^\vee/m)}
\end{align*} and 
$$S_2(m,p)=\E\left(e^{iu\langle \frac{1}{m}\Pi^m\lfloor mt\rfloor,\alpha_1^\vee\rangle}(1_{\{\xi_{p+1}^m(\infty)=0\}}-1_{\{\xi_{p+1}^m\lfloor mt\rfloor=0\}})\right)/\P\left(\xi_{p+1}^m(\infty)=0\right).$$
The convergence of $\frac{1}{m} \Pi_+^m\lfloor mt\rfloor$ towards $A(t)$ when $m$ goes to infinity, and Lemma \ref{convD} imply that 
$$\lim_{m\to \infty} S_1(u,m,p)=\E\left(\psi_p(u,A(t))\right)$$
where for $\lambda\in C_{\mbox{aff}}$, $$\psi_p(u,\lambda)=\frac{e^{iu\langle \lambda,\alpha^\vee_1\rangle }}{2\varphi_{1/2}(\lambda)}\E\left(e^{-iu\langle L_p,\alpha^\vee_1\rangle}1_{\xi_{\cdot,p}\in \Gamma(\lambda)}\right).$$
Lemmas \ref{conv-cryst}   and Propositions \ref{string-cond} and \ref{EcondB} imply that 
$$\lim_{p\to \infty}\E\left(\psi_p(u,A(t))\right)=\E\left(e^{iu\langle B(t),\alpha^\vee_1\rangle }\right)$$
As $\{\xi_{p+1}^m(\infty)=0\}\subset \{\xi_{p+1}^m\lfloor mt\rfloor=0\}$
one has, $$\left\vert S_2(m,p)\right\vert\le \frac{\P(\xi_{p+1}^m\lfloor mt\rfloor=0)}{\P(\xi_{p+1}^m(\infty)=0)}-1=S_1(0,m,p)-1$$
which implies that $\lim_{p\to\infty}\lim_{m\to\infty} S_2(m,p)=0$ and ends the proof of the convergence  in law  of  $I_{\alpha_0}^{\xi_{0,p}(\infty)}\dots I_{\alpha_p}^{\xi_{p,p}(\infty)}A(t)$         towards $B(t)$ when $p$ goes to infinity. 

Let now $t_0,\dots,t_n$ be a sequence of  ordered real numbers such that $0=t_0<t_1<\dots<t_n$, and $u_1,\dots,u_n\in \R$. For   $m,p\ge 1$, the Fourier transform
\begin{align}\label{law-vector} \E\left(e^{i\sum_{k=1}^nu_k\big(\langle I_{\alpha_0}^{\xi_{0,p}(\infty)}\dots I_{\alpha_p}^{\xi_{p,p}(\infty)}A(t_k),\alpha^\vee_1\rangle-\langle I_{\alpha_0}^{\xi_{0,p}(\infty)}\dots I_{\alpha_p}^{\xi_{p,p}(\infty)}A(t_{k-1}),\alpha^\vee_1\rangle\big)}\right)
\end{align}
is equal to 
$$\lim_{m\to\infty} \E\left(e^{i\sum_{k=1}^n\frac{u_k}{m}\langle I_{\alpha_0}^{\xi^m_{0,p}(\infty)}\dots I_{\alpha_p}^{\xi^m_{p,p}(\infty)}\Pi_+\lfloor mt_k\rfloor-I_{\alpha_0}^{\xi^m_{0,p}(\infty)}\dots I_{\alpha_p}^{\xi^m_{p,p}(\infty)}\Pi_+\lfloor mt_{k-1}\rfloor,\alpha^\vee_1\rangle}\right).$$
We obtain as previously, introducing this time the event $\{\xi_{p+1}^m\lfloor mt_1\rfloor=0\}$ and using the independence of the increments, that  the Fourier transform   (\ref{law-vector}) 
converges towards 
$$\E\left(e^{i\sum_{k=1}^nu_k \langle B(t_k)-B(t_{k-1}),\alpha_1^\vee\rangle}\right),$$
when $p$ goes to infinity, which ends the proof.

  \end{proof}

\end{document}